\documentclass{amsart}
\usepackage[english]{babel}
\usepackage{amssymb,enumerate,latexsym}

\newenvironment{enumerateroman}{\begin{enumerate}[\textup{(}i\textup{)}] }{\end{enumerate}}
\newenvironment{proof*}[1]{\noindent\textit{#1.}}{\hspace*{\fill}$\Box$\medskip}
\newtheorem{lemma}{Lemma}
\newtheorem{theorem}{Theorem}
\newtheorem{remark}{Remark}
\newtheorem{thmA}{Theorem}

\newtheorem*{conj}{Chern Conjecture}
\newtheorem*{conjg}{Generalized Chern Conjecture}
\newtheorem*{problem}{Open Problem}

\newcommand{\sumik}{\underset{k \geq 2}{\sum_i}}

\begin{document}

\title[A new gap for hypersurfaces in space forms]{A new gap for complete hypersurfaces with constant mean curvature in space forms}
\thanks{
\indent Research supported by the National Natural Science Foundation of China, Grant Nos. 11531012, 11371315, 11771394; and the China Postdoctoral Science Foundation, Grant No. BX20180274.}
\author{Juanru Gu}

\author{Li Lei}

\author{Hongwei Xu}

\begin{abstract}
  Let $M$ be an $n$-dimensional closed hypersurface with
  constant mean curvature and constant scalar curvature in an unit sphere.
  Denote by $H$ and $S$ the mean curvature and the squared length of the
  second fundamental form respectively. We prove that if $S > \alpha (n, H)$, where $n\geq 4$ and $H\neq 0$,
  then $S > \alpha (n, H) + B_n\frac{n H^2}{n - 1}$. Here \[ \alpha (n, H) = n + \frac{n^3}{2 (n - 1)} H^2 - \frac{n (n - 2)}{2 (n - 1)}
   \sqrt{n^2 H^4 + 4 (n - 1) H^2}, \]  $B_n=\frac{1}{5}$ for $4\leq n \leq 20$, and
  $B_n=\frac{49}{250}$  for $n>20$. Moreover, we
  obtain a gap theorem for complete hypersurfaces with
  constant mean curvature and constant scalar curvature in space forms.
\end{abstract}

{\maketitle}

\section{Introduction}

In the late 1960's, Simons \cite{Simons1968}, Lawson \cite{Lawson1969} and Chern-do
Carmo-Kobayashi \cite{CdK1970} proved that if $M$ is a closed minimal
hypersurface in $\mathbb{S}^{n + 1}$, whose squared length of the second
fundamental form satisfies $S \leq n$, then $S \equiv 0$ and $M$ is the
great sphere $\mathbb{S}^n$, or $S \equiv n$ and $M$ is the Clifford torus.
Moreover, they obtained a rigidity theorem for $n$-dimensional closed minimal
submanifolds in $\mathbb{S}^{n + p}$ under the pinching condition $S \leq
\frac{n}{2 - 1 / p}$. Afterwards, Li-Li \cite{LL1992} improved Simons' pinching constant for
closed minimal submanifolds to $\max \{ \frac{n}{2 - 1 / p}, \frac{2
n}{3} \}$. For minimal submanifolds in spheres, constant
scalar curvature implies constant length of second fundamental form.
In 1970's, Chern proposed the following famous conjecture
\cite{Chern1968,CdK1970}, which was listed in the well-known problem section by Yau
\cite{Yau1982}.

\begin{conj}
  Let $M$ be a closed
minimal hypersurface with constant scalar curvature in the unit sphere
$\mathbb{S}^{n + 1}$. Then the set of all possible values of the
scalar curvature of $M$ is discrete.
\end{conj}

Based on the fact that isoparametric hypersurfaces are the only known examples
of closed minimal hypersurfaces with constant scalar curvature in a sphere,
and the scalar curvatures of these isoparametric hypersurfaces have
only finite number of values, mathematicians conjectured that closed minimal
hypersurfaces with constant scalar curvature in spheres must be isoparametric.
In \cite{Munzner1980}, M\"unzner showed that for a closed
isoparametric minimal hypersurface in an unit sphere, the number
of distinct principal curvatures satisfies
$g\in\{1,\,2,\,3,\,4,\,6\}$, and $S=(g-1)n$.

During the last three decades, there have been many important progresses on
Chern Conjecture (see \cite{Cheng1997,GT2012,GXXZ2016,SWY} for more details).
In 1983, Peng and Terng \cite{PT1983a} made a breakthrough on the Chern conjecture.
Peng-Terng proved that if $M^n$ is a closed minimal hypersurface with constant
scalar curvature in $\mathbb{S}^{n + 1}$, and if $S > n$, then $S > n +
\frac{1}{12 n}$. In particular, for the case $n = 3$, they verified that if $S
> 3$, then $S \geq 6$.
Afterwards, Yang-Cheng \cite{YC1990,YC1994,YC1998}
improved the pinching constant $\frac{1}{12 n}$ to $\frac{n}{3}$, Suh-Yang
\cite{SY2007} improved it to $\frac{3 n}{7}$. In 1993, Chang \cite{Chang1993} solved
Chern conjecture in dimension three.
Recently, Deng-Gu-Wei \cite{DGW2017} proved that any closed Willmore minimal hypersurface with
constant scalar curvature in $\mathbb{S}^5$ must be isoparametric.
For closed minimal hypersurfaces in
spheres, the scalar curvature pinching phenomenon without the assumption of
constant scalar curvature was investigated by many authors
\cite{CI1999,DX2011,LXX2017,PT1983b,WX2007,XX2017,Zhang2010}.

Let $M$ be an $n$-dimensional closed submanifold in a sphere. Denote by $H$ and
$S$ its mean curvature and the squared norm of its second fundamental form
respectively. For hypersurfaces with constant mean curvature, we have the
following generalized version of Chern's conjecture.

\begin{conjg}
  Let $M$ be a
closed hypersurface with constant mean curvature and constant scalar curvature
in the unit sphere $\mathbb{S}^{n + 1}$. Then for each $n$ and $H$, the set of
all possible values of $S$ is discrete.
\end{conjg}

We set
\[ \alpha (n, H) = n + \frac{n^3}{2 (n - 1)} H^2 - \frac{n (n - 2)}{2 (n - 1)}
   \sqrt{n^2 H^4 + 4 (n - 1) H^2} . \]
Cheng-Nakagawa \cite{CN1990} and Xu \cite{Xu1990} got the following rigidity theorem for closed hypersurfaces with constant mean curvature in a sphere.

\begin{thmA}
  Let $M$ be an $n$-dimensional closed hypersurface
with constant mean curvature in the unit sphere $\mathbb{S}^{n + 1}$. If the
squared length of the second fundamental form satisfies $S \leq \alpha (n, H)$,
then $M$ is either a totally umbilical sphere, or a Clifford torus.
\end{thmA}

More generally, the third author \cite{Xu1990,Xu1993a} proved the generalized
Simons-Lawson-Chern-do Carmo-Kobayashi theorem for compact submanifolds with
parallel mean curvature in a sphere.

In 1990, de Almeida and Brito \cite{dB1990} proved that closed 3-dimensional hypersurfaces with constant mean curvature and nonnegative constant scalar curvature in a space form must be isoparametric.
In \cite{Chang1993b}, Chang  proved that if $M$ is a closed hypersurface with constant mean curvature and constant scalar curvature in the unit sphere $\mathbb{S}^4$, then $M$ is isoparametric.
Chang \cite{Chang1994} also proved that a closed hypersurface with constant mean curvature, constant scalar curvature and three distinct principal curvatures in a sphere must be isoparametric.
Later, Cheng and Wan \cite{CW1994} classified complete hypersurfaces with constant
mean curvature and constant scalar curvature in $\mathbb{R}^4$. In 2017, N$\acute{u}\tilde{n}$ez \cite{Nunez} investigated the classification problem for closed hypersurfaces with constant mean curvature and constant scalar curvature in $\mathbb{R}^5$.
Recently, Tang-Wei-Yan \cite{TWY2018} generalized the theorem of de Almeida and Brito \cite{dB1990} to higher dimensional cases.

Put $$\beta(n,H)=n+\frac{n^3}{2(n-1)}H^2+\frac{n(n-2)}{2(n-1)}\sqrt{n^2H^4+4(n-1)H^2}.$$
Xu-Tian \cite{XT2011} proved the following second gap theorem for hypersurfaces with constant mean curvature and constant scalar curvature in a sphere, which generalized Suh-Yang's second gap theorem for minimal hypersurfaces.
\begin{thmA} Let $M$ be a
compact hypersurface with constant mean curvature and constant scalar curvature in the unit sphere $\mathbb{S}^{n+1}$.  There exists a
positive constant $\gamma(n)$ depending only on $n$ such that if
$|H|<\gamma(n)$, and $\beta(n,H)\leq
S\leq\beta(n,H)+\frac{3n}{7}$, where $n\geq 4$ and $H\neq 0$, then $S\equiv\beta(n,H)$ and
$M=\mathbb{S}^{1}(\frac{1}{\sqrt{1+\mu^2}})\times
\mathbb{S}^{n-1}(\frac{\mu}{\sqrt{1+\mu^2}})$. Here
 $\mu=\frac{n|H|+\sqrt{n^2H^2+4(n-1)}}{2}$.\end{thmA}

For hypersurfaces with
small constant mean curvature in a sphere, there have been several scalar curvature pinching theorems
\cite{GXXZ2016,XX2013,XX2014}.

When $H\neq 0$, we set
 $$\alpha_k(n,H)=n+\frac{n^3}{2k(n-k)}H^2-\frac{n(n-2k)}{2k(n-k)}\sqrt{n^2H^4+4k(n-k)H^2},$$
for $k=1,\cdots, n-1$. Then we have
 \begin{eqnarray*}&&\alpha(n,H)=\alpha_1(n,H)<\cdots<\alpha_{n-1}(n,H)=\beta(n,H),\end{eqnarray*}
and the isoparametric hypersurfaces $\mathbb{S}^{n-k}(\frac{1}{\sqrt{1+\lambda_k^2}})\times \mathbb{S}^{k}(\frac{\lambda_k}{\sqrt{1+\lambda_k^2}})$ satisfy $$S\equiv\alpha_{k}(n,H),$$ where $\lambda_k=\frac{n|H|+\sqrt{n^2H^2+4k(n-k)}}{2(n-k)}$ for $k=1,2,...,n-1$.

It is well known that the possible values of the squared length of the second
fundamental forms of all closed isoparametric hypersurfaces with constant mean
curvature in the unit sphere form a discrete set, which was explicitly
given by Muto \cite{Muto1988}.
 In 2014, Xu-Xu \cite{XX2014} showed that there exists a compact isoparametric hypersurface with 3 distinct principal curvatures in the unit sphere $\mathbb{S}^{3k+1}$ satisfying $S=2n+3nH^2$ for $n=3k$, where $k=2,4,8$.
  A direct computation shows that $$\alpha_{[\frac{n}{2}]}(n,H)<2n+3nH^2.$$
  Motivated by Theorem A and the inequality above, we have the following open problem.
\begin{problem} Let $M$ be a compact hypersurface with constant mean curvature and constant scalar curvature in the unit sphere $\mathbb{S}^{n+1}$. Assume that $$\alpha(n, H)\leq S\leq\alpha_2(n,H),$$  where $n\geq 4$ and $H\neq0$. Is it possible to prove that $M$ must be one of the isoparametric hypersurfaces $\mathbb{S}^{n-k}(\frac{1}{\sqrt{1+\lambda_k^2}})\times \mathbb{S}^{k}(\frac{\lambda_k}{\sqrt{1+\lambda_k^2}})$ for $k=1, 2$? Here $\lambda_k=\frac{n|H|+\sqrt{n^2H^2+4k(n-k)}}{2(n-k)}$.\end{problem}

In this paper, we first verify the third gap theorem for hypersurfaces in a sphere.
\begin{theorem}
  \label{Thm1}Let $M$ be a closed
  hypersurface with constant mean curvature and constant scalar curvature in the unit sphere
  $\mathbb{S}^{n + 1}$. If $S > \alpha (n, H)$,  where $n\geq 4$ and $H\neq 0$, then $$S >\alpha (n,
  H) + B_n\frac{n H^2}{n - 1}.$$ Here $$B_n=\left\{\begin{array}{llll} \frac{1}{5},\mbox{\hspace*{6mm} $for$\ }4\leq n \leq 20,\\
  \frac{49}{250},\mbox{\hspace*{3mm} $for$\ }n>20.\\
 \end{array} \right.$$
\end{theorem}

 Set
\[ \alpha (n, H, c) = n c + \frac{n^3}{2 (n - 1)} H^2 - \frac{n (n - 2)}{2 (n
   - 1)} \sqrt{n^2 H^4 + 4 (n - 1) c H^2} , \]
and $\mathring{\alpha} (n, H, c) = \alpha (n, H, c) - n H^2$.  More generally, we prove the following gap theorem for complete hypersurfaces in space forms.

\begin{theorem}
  \label{Thm2}Let $M$ be a complete
  hypersurface with constant mean curvature and constant scalar curvature in the
  simply connected space form $\mathbb{F}^{n + 1} (c)$ with $c + H^2 > 0$.
  If
  \[ 0 \leq S - \alpha (n, H, c) \leq B_n \min \left\{ \frac{nH^2}{n -
     1}, \mathring{\alpha} (n, H, c) \right\}, \]
  where $n\geq 4$ and $H\neq 0$, then $S = \alpha (n, H, c)$ and $M$ is the isoparametric hypersurface
  $\mathbb{F}^{n - 1} (c + \lambda^2) \times \mathbb{F}^1 (c + c^2 \lambda^{-
  2})$. Here $$\lambda = \frac{n | H | + \sqrt{n^2 H^2 + 4 (n - 1) c}}{2 (n -
  1)},$$
  and $B_n$ is the same as in Theorem 1.
\end{theorem}
\begin{remark} Notice that $\frac{n}{n - 1} H^2 < \mathring{\alpha} (n, H, c)$ if $c > 0$, and
$\frac{n}{n - 1} H^2 \geq \mathring{\alpha} (n, H, c)$ if $c \leq 0$.
Thus, Theorem \ref{Thm2} implies Theorem \ref{Thm1}.\end{remark}

\

\section{hypersurfaces with constant mean curvature}

Let $\mathbb{F}^{n + 1} (c)$ be the ($n + 1$)-dimensional simply connected
space form with constant sectional curvature $c$, and let $M^n$ be a
hypersurface in $\mathbb{F}^{n + 1} (c)$. We denote by $\overline{\nabla}$ the
Levi-Civita connection of $\mathbb{F}^{n + 1} (c)$, and by $\nabla$ the
connection induced on $M$. Let $\nu$ be a unit normal vector field of $M$.
Denote by $h$ the second fundamental form of $M$, which is a symmetric
bilinear form given by
\[ \overline{\nabla}_X Y = \nabla_X Y + h (X, Y) \nu .\]
The shape operator $A : T M \rightarrow
T M$ is defined by
\[ A (X) = - \overline{\nabla}_X \nu . \]
The second fundamental form and the shape operator are related by $h (X, Y) =
\langle A (X), Y \rangle$.
The eigenvalues of $A$ are called the principal curvatures of $M$. If the
principal curvatures are all constant, then $M$ is called an isoparametric
hypersurface.

We choose a local orthonormal frame $\{ e_i \}$ for the tangent bundle of $M$.
Set $h_{i j} = h (e_i, e_j)$. The mean curvature is given by $H = \frac{1}{n}
\sum_i h_{i i}$. Denote by $S$ the squared length of the second fundamental
form, i.e. $S = \sum_{i, j} h_{i j}^2$. We denote by $h_{i j k}$ the covariant
derivative of $h_{i j}$. Then the Codazzi equation implies that $h_{i j k}$ is
symmetric in $i, j$ and $k$.

Suppose that $M$ has constant mean curvature. Then we have the following
Simons' type formulas.
\begin{equation}
  (\Delta h)_{i j} = - n cH \delta_{i j} + (n c - S) h_{i j} + n H \sum_k h_{i
  k} h_{j k}, \label{Laph}
\end{equation}
\begin{equation}
  \frac{1}{2} \Delta S = | \nabla h |^2 + S (n c - S) - n^2 c H^2 + n H
  \operatorname{tr} A^3 . \label{LapS}
\end{equation}

Now we suppose that $\lambda$ is a principal curvature whose multiplicity is 1
at some point $x \in M$. Then $\lambda$ is smooth in a neighborhood of $x$.
Let $u$ be the unit eigenvector corresponding to $\lambda$, i.e.
\begin{equation}
  \lambda u = A (u) . \label{eigen}
\end{equation}
Differentiating {\eqref{eigen}} with respect to a tangent vector $X$, we get
\begin{equation}
  X (\lambda) u + \lambda \nabla_X u = \nabla_X A (u) + A (\nabla_X u) .
  \label{deigen}
\end{equation}
Let $V$ be the orthogonal complement of $u$ in tangent space $T_x M$. Then $V$
is an ($n - 1$)-dimensional $A$-invariant subspace. Since $| u | = 1$, we have
$\nabla_X u \in V$. Thus $A (\nabla_X u) \in V$. Let $(\cdot)^V$ denote the
projection onto $V$. From {\eqref{deigen}} we obtain
\begin{equation}
  X (\lambda) = \langle \nabla_X A (u), u \rangle = \nabla h (u, u, X)
  \label{deigen1}
\end{equation}
and
\begin{equation}
  (\lambda \operatorname{id} - A) (\nabla_X u) = [\nabla_X A (u)]^V .
\end{equation}
Since  $\lambda \operatorname{id}_V - A|_V$ is invertible, we get
\[ \nabla_X u = (\lambda \operatorname{id}_V - A|_V)^{- 1} ([\nabla_X A (u)]^V) . \]
Taking covariant derivative of left and right hand sides
of {\eqref{deigen1}} with respect to $Y$, we get
\begin{eqnarray*}
  \nabla^2 \lambda (X, Y) & = & \nabla^2 h (u, u, X, Y) + 2 \nabla h (\nabla_Y
  u, u, X)\\
  & = & \nabla^2 h (u, u, X, Y) + 2 \nabla h ((\lambda \operatorname{id}_V - A|_V)^{-
  1} ([\nabla_Y A (u)]^V), u, X) .
\end{eqnarray*}
Taking trace, we obtain
\begin{equation}
  \Delta \lambda = \Delta h (u, u) + 2 \sum_i \nabla h ((\lambda \operatorname{id}_V -
  A|_V)^{- 1} ([\nabla_{e_i} A (u)]^V), u, e_i) \label{Laplam0}
\end{equation}
at point $x$. We choose the orthonormal frame $\{ e_i \}$ at $x$, such that
$h_{i j} = \lambda_i \delta_{i j}$, $\lambda_1 = \lambda$ and $e_1 = u$. Then
formula {\eqref{Laplam0}} can be written as
\[ \Delta \lambda_1 = (\Delta h)_{11} + 2 \sumik  \frac{h_{1 i k}^2}{\lambda_1
   - \lambda_k} . \]
This together with {\eqref{Laph}} yields
\begin{equation}
  \Delta \lambda_1 = - n cH + (n c - S) \lambda_1 + n H \lambda_1^2 + 2 \sumik
  \frac{h_{1 i k}^2}{\lambda_1 - \lambda_k} . \label{Laplam}
\end{equation}

\

\section{some algebraic inequalities}

Let $\mu_1 \leq \cdots \leq \mu_n$ be $n$ real numbers, which
satisfy
\[ \sum_i \mu_i = 0 \qquad \operatorname{and} \qquad \sum_i \mu_i^2 = 1. \]
Set
\begin{equation}
  \phi = \sum_i \mu_i^3 + \frac{n - 2}{\sqrt{n (n - 1)}}, \quad \eta =
  \sqrt{\frac{n}{n - 1}} \mu_1 + 1, \quad \sigma = \left[ \sum_{i \geq 2}
  \left( \mu_i + \frac{\mu_1}{n - 1} \right)^2 \right]^{\frac{1}{2}} .
  \label{fhs}
\end{equation}
\begin{lemma}
  \label{mui}The functions $\phi$, $\eta$ and $\sigma$ satisfy
  \begin{enumerateroman}
    \item $\frac{\sqrt{n (n - 1)}}{n - 2} \phi \geq \eta \geq
    \frac{\sigma^2}{2}$,

    \item $\sqrt{n (n - 1)} \phi \geq \eta [3 n - 3 (n + 1) \eta - 2\sqrt{n(n-1)}
    \sigma]$.
  \end{enumerateroman}
\end{lemma}

\begin{proof}
  By the definitions, we have
  \[ \sigma^2 = \sum_{i \geq 2} \left( \mu_i + \frac{\mu_1}{n - 1}
     \right)^2 = 1 - \frac{n}{n - 1} \mu_1^2 = \eta (2 - \eta) . \]
  Since $\mu_1 < 0$, we get $\eta < 1$. Then we have
  \[ \eta = 1 - \sqrt{1 - \sigma^2} \geq \frac{\sigma^2}{2} . \]

  By a direct computation, we get
  \begin{eqnarray}
    &  & \sum_i \left( \mu_i + \frac{\mu_1}{n - 1} \right)^2 (\mu_i - \mu_1)
    \nonumber\\
    & = & \sum_i \mu_i^3 - \frac{\mu_1}{n - 1} \left( n - 3 + \frac{n}{n - 1}
    \mu_1^2 \right)  \label{phi-}\\
    & = & \phi - \frac{n - 2}{\sqrt{n (n - 1)}} \eta + \frac{\mu_1}{n - 1}
    \sigma^2 . \nonumber
  \end{eqnarray}
  Thus we have
  \[ \phi \geq \frac{n - 2}{\sqrt{n (n - 1)}} \eta . \]
  Hence we prove inequality (i).

  From {\eqref{phi-}}, we have
  \begin{eqnarray}
    \phi & \geq & \frac{n - 2}{\sqrt{n (n - 1)}} \eta - \frac{\mu_1}{n -
    1} \sigma^2 + \sum_{i \geq 2} \left( \mu_i + \frac{\mu_1}{n - 1}
    \right)^2 (\mu_2 - \mu_1) \nonumber\\
    & = & \frac{n - 2}{\sqrt{n (n - 1)}} \eta - \frac{\mu_1}{n - 1} \sigma^2
    + \sigma^2 (\mu_2 - \mu_1)  \label{phi>=}\\
    & = & \frac{n - 2}{\sqrt{n (n - 1)}} \eta + \sigma^2 \left( \mu_2 -
    \frac{n}{n - 1} \mu_1 \right) . \nonumber
  \end{eqnarray}
  Since $\left| \mu_2 + \frac{\mu_1}{n - 1} \right| \leq \sigma$, we have
  \begin{eqnarray}
    \sigma^2 \left( \mu_2 - \frac{n}{n - 1} \mu_1 \right) & \geq &
    \sigma^2 \left( - \frac{1}{n - 1} \mu_1 - \sigma - \frac{n}{n - 1} \mu_1
    \right) \nonumber\\
    & = & \eta (2 - \eta) \left( \frac{n + 1}{\sqrt{n (n - 1)}} (1 - \eta) -
    \sigma \right)  \label{sig2mu2-}\\
    & \geq & \eta \left[ \frac{n + 1}{\sqrt{n (n - 1)}} (2 - 3 \eta) - 2
    \sigma \right] \nonumber\\
    & = & \frac{1}{\sqrt{n (n - 1)}} \eta [(n + 1)(2 - 3 \eta) - 2\sqrt{n(n-1)} \sigma]
    . \nonumber
  \end{eqnarray}
  Substituting {\eqref{sig2mu2-}} into {\eqref{phi>=}}, we obtain conclusion (ii).
\end{proof}

\begin{lemma}
  \label{phi<1/16sqrt} Let $n \geq 4$ and \[B_n=\left\{\begin{array}{llll} \frac{1}{5},\mbox{\hspace*{6mm} $for$\ }4\leq n \leq 20,\\
  \frac{49}{250},\mbox{\hspace*{3mm} $for$\ }n>20.\\
  \end{array} \right.\]
  If $\phi \leq \frac{B_n}{2}
  \sqrt{\frac{n}{n - 1}}$, then $\eta < 0.0445$, $\sigma <
  0.295$ and $\mu_2 - \mu_1 > \tfrac{2}{3-3^{-9}} \sqrt{\tfrac{n}{n -
        1}}$.
\end{lemma}

\begin{proof}
Since $\sigma=\sqrt{\eta(2-\eta)}\leq \sqrt{2
     \eta}$, we get from Lemma \ref{mui} (ii) that
 \[ \sqrt{\tfrac{n - 1}{n}} \phi
      \geq \eta \left( 3 - \frac{3(n+1)}{n}
     \eta -2\sqrt{\frac{2(n-1)\eta}{n}}  \right) .\]
  Let
  \[ f (\eta) := \eta \left( 3 - \frac{3(n+1)}{n}
     \eta -2\sqrt{\frac{2(n-1)\eta}{n}}\right). \]
  Then we have $f''(\eta) < 0$.

 Case (i). If $4\leq n\leq20$, then we have
 \[ f (\eta) \geq \eta \left( 3 - \frac{15}{4}
     \eta - \sqrt{\frac{38}{5}\eta}\right). \]
 We get from Lemma \ref{mui} (i) that
  \[ \eta \leq \frac{\sqrt{n (n - 1)}}{n - 2} \phi \leq \frac{n}{10 (n - 2)}
     \leq \frac{1}{5}  .\]
  Assume $\eta \geq 0.0445$. From $0.0445 \leq \eta
  \leq 0.2$, we get
  \[ f (\eta) \geq \min \left\{ f \left( 0.0445 \right), f \left(
    0.2 \right) \right\} > \frac{1}{10}. \]
  This contradicts the condition $\sqrt{\tfrac{n - 1}{n}} \phi \leq
  \frac{1}{10}$. Therefore, we have $\eta < 0.0445$ and $\sigma
  =\sqrt{\eta(2-\eta)} < 0.295$.

  Then we have
  \begin{eqnarray}
  \mu_2 - \mu_1 & \geq & - \frac{\mu_1}{n - 1} - \sigma - \mu_1
  \nonumber\\
  & = & \sqrt{\tfrac{n}{n - 1}} (1 - \eta) - \sigma  \label{muk-mu1}\\
  & \geq & \sqrt{\tfrac{n}{n - 1}} (1 - \eta - \sqrt{\tfrac{19}{20}}\sigma) \nonumber\\
  & > & 0.667 \sqrt{\tfrac{n}{n - 1}} . \nonumber
  \end{eqnarray}

Case (ii). If $ n>20$, then we have
 \[ f (\eta) \geq \eta \left( 3 - \frac{66}{21}
     \eta - 2\sqrt{2\eta}\right). \]
 We also get from Lemma \ref{mui} (i) that
  \[ \eta \leq \frac{\sqrt{n (n - 1)}}{n - 2} \phi \leq\frac{49n}{500(n - 2)}
     \leq 0.11  .\]
  Assume $\eta \geq 0.04305$. From $0.04305 \leq \eta
  \leq 0.11$, we get
  \[ f (\eta) \geq \min \left\{ f \left( 0.04305 \right), f \left(
    0.11 \right) \right\} > 0.098. \]
  This contradicts the condition $\sqrt{\tfrac{n - 1}{n}} \phi \leq
  \frac{49}{500}$. Therefore, we have $\eta < 0.04305$ and $\sigma
  =\sqrt{\eta(2-\eta)}< 0.29026$.

Then we get
    \[\mu_2 - \mu_1
     \geq  \sqrt{\tfrac{n}{n - 1}} (1 - \eta -\sigma)
    >\tfrac{2}{3-3^{-9}} \sqrt{\tfrac{n}{n - 1}}.
  \]
  Combining the two cases, we prove Lemma \ref{phi<1/16sqrt}.
\end{proof}

For integer $n > 1$, and real numbers $c, H$, we define
\[ \alpha (n, H, c) = n c + \frac{n^3}{2 (n - 1)} H^2 - \frac{n (n - 2)}{2 (n
   - 1)} \sqrt{n^2 H^4 + 4 (n - 1) c H^2}, \]
\begin{equation}
  \mathring{\alpha} (n, H, c) = \alpha (n, H, c) - n H^2 . \label{alp}
\end{equation}
\begin{lemma}
  \label{idealp}If $H^2 + c > 0$, then
  \[ (n - 2) \sqrt{\frac{n}{n - 1} H^2  \mathring{\alpha} (n, H, c)} = n (H^2 + c)
     - \mathring{\alpha} (n, H, c) . \]
\end{lemma}

\begin{proof}
  Note that
  \begin{equation} n (H^2 + c) - \mathring{\alpha} (n, H, c) = \frac{n (n - 2)}{2 (n - 1)} \left[
     - (n - 2) H^2 + \sqrt{n^2 H^4 + 4 (n - 1) c H^2} \right] . \label{nH2+c} \end{equation}
  Since
  \[ [n^2 H^4 + 4 (n - 1) c H^2] - [(n - 2) H^2]^2 = 4 (n - 1) H^2 (H^2 + c)
     \geq 0, \]
  we get
  \[ n (H^2 + c) - \mathring{\alpha} (n, H, c) \geq 0. \]
  It follows from \eqref{nH2+c} that
  \begin{eqnarray*}
    &  & [n (H^2 + c) - \mathring{\alpha} (n, H, c)]^2\\
    & = & \frac{n^2 (n - 2)^2}{2 (n - 1)^2} H^2 \left[ 2 (n - 1) c + (n^2 - 2
    n + 2) H^2 - (n - 2) \sqrt{n^2 H^4 + 4 (n - 1) c H^2} \right]\\
    & = & \frac{n (n - 2)^2}{n - 1} H^2  \mathring{\alpha} (n, H, c) .
  \end{eqnarray*}
  This proves Lemma \ref{idealp}.
\end{proof}

\

\section{Proof of the main theorem}

Suppose $M$ is an $n (\geq 4)$-dimensional hypersurface in $\mathbb{F}^{n
+ 1} (c)$ with constant $H$ and constant $S$. Let $\mathring{S}$ denote the squared
length of the traceless second fundamental form, i.e. $\mathring{S} = S - n H^2$.
Let $\mathring{\alpha} (n, H, c)$ be the constant given by {\eqref{alp}}. For ease
of notation, we write $\mathring{\alpha}$ instead of $\mathring{\alpha} (n, H, c)$. We
choose the orthonormal frame $\{ e_i \}$, such that $h_{i j} = \lambda_i
\delta_{i j}$, where the principal curvatures $\lambda_i,1 \leq i \leq n$
satisfy $\lambda_1 \leq \cdots \leq \lambda_n$.
 Assuming $\mathring{S} > 0$, we put $\mu_i =
(\lambda_i - H) \mathring{S}^{- 1 / 2}$. Let $\phi$, $\eta$ and $\sigma$ be
functions of $\mu_i$ defined as {\eqref{fhs}}.\\

\begin{proof*}{Proof of Theorem \ref{Thm2}}
  Without loss of generality, we assume $H > 0$. Letting $\delta = B_n
  \min \left\{ \frac{n}{n - 1} H^2, \mathring{\alpha} \right\}$, we have
  $\mathring{\alpha} \leq \mathring{S} \leq \mathring{\alpha} + \delta$.

  Since $S$ is constant, the formula {\eqref{LapS}} implies
  \begin{eqnarray*}
    0 & = & | \nabla h |^2 + S (n c - S) - n^2 c H^2 + n H \sum_i
    \lambda_i^3\\
    & = & | \nabla h |^2 - \mathring{S}^2 + n \mathring{S} (H^2 + c) + n H \mathring{S}^{3 /
    2} \sum_i \mu_i^3.
  \end{eqnarray*}
  This implies
  \begin{equation}
    | \nabla h |^2 + n H \mathring{S}^{3 / 2} \phi = \mathring{S} \left[ \mathring{S} - n
    (H^2 + c) + (n - 2) \sqrt{\frac{n}{n - 1}  \mathring{S}} H \right] .
    \label{dh2+nHS}
  \end{equation}
  By Lemma \ref{idealp} and the definition of $\delta$, we get
  \begin{eqnarray}
    &  & \mathring{S} - n (H^2 + c) + (n - 2) \sqrt{\frac{n}{n - 1}  \mathring{S}} H
    \nonumber\\
    & \leq & \mathring{\alpha} + \delta - n (H^2 + c) + (n - 2) \sqrt{\frac{n}{n -
    1}} \left( \sqrt{\mathring{\alpha}} + \frac{\delta}{2 \sqrt{\mathring{\alpha}}}
    \right) H \nonumber\\
    & = & \delta + (n - 2) \sqrt{\frac{n}{n - 1}}  \frac{\delta}{2
    \sqrt{\mathring{\alpha}}} H  \label{n/16sqrt}\\
    & \leq & B_n \sqrt{\frac{n}{n - 1} H^2  \mathring{\alpha}} + B_n(n - 2)
    \sqrt{\frac{n}{n - 1}}  \frac{\sqrt{\mathring{\alpha}}}{2} H \nonumber\\
    & \leq & \frac{B_nn}{2} \sqrt{\frac{n}{n - 1}  \mathring{S}} H. \nonumber
  \end{eqnarray}
  From {\eqref{dh2+nHS}} and {\eqref{n/16sqrt}}, we obtain
  \[ \phi \leq \frac{B_n}{2} \sqrt{\frac{n}{n - 1}} . \]
  Using Lemma \ref{phi<1/16sqrt}, we obtain that the smallest principal
  curvature $\lambda_1$ has multiplicity one. Thus the equation
  {\eqref{Laplam}} is valid everywhere on $M$. Inserting $\lambda_i = \mu_i
  \mathring{S}^{\frac{1}{2}} + H$ and $\mu_1 = \sqrt{\tfrac{n - 1}{n}} (\eta - 1)$
  into {\eqref{Laplam}}, we have
  \begin{eqnarray*}
    \Delta \lambda_1 & = & \frac{2}{\sqrt{\mathring{S}}} \sumik  \frac{h_{1 i
    k}^2}{\mu_1 - \mu_k} - n cH\\
    &  & + (n c - S) \left[ \sqrt{\tfrac{n - 1}{n} \mathring{S}} (\eta - 1) + H
    \right] + n H \left[ \sqrt{\tfrac{n - 1}{n} \mathring{S}} (\eta - 1) + H
    \right]^2\\
    & = & \frac{2}{\sqrt{\mathring{S}}} \sumik  \frac{h_{1 i k}^2}{\mu_1 -
    \mu_k}\\
    &  & + \eta \sqrt{\mathring{S}} \left[ \sqrt{\tfrac{n - 1}{n}} (n (H^2 + c) -
    \mathring{S}) + (n - 1) (\eta - 2) H \sqrt{\mathring{S}} \right]\\
    &  & + \sqrt{\tfrac{n - 1}{n} \mathring{S}} \left[ \mathring{S} - n (H^2 + c) + (n
    - 2) \sqrt{\tfrac{n}{n - 1}  \mathring{S}} H \right] .
  \end{eqnarray*}
  This together with {\eqref{dh2+nHS}} yields
  \begin{eqnarray}
    \Delta \lambda_1 & = & \frac{2}{\sqrt{\mathring{S}}} \sumik  \frac{h_{1 i
    k}^2}{\mu_1 - \mu_k} + \sqrt{\frac{n - 1}{n \mathring{S}}} | \nabla h |^2 +
    \sqrt{n (n - 1)} H \mathring{S} \phi \nonumber\\
    &  & + \eta \sqrt{\mathring{S}} \left[ \sqrt{\tfrac{n - 1}{n}} (n (H^2 + c) -
    \mathring{S}) + (n - 1) (\eta - 2) H \sqrt{\mathring{S}} \right] .  \label{Laph11=}
  \end{eqnarray}
  By Lemma \ref{phi<1/16sqrt}, we have
  \begin{eqnarray}
    \sumik  \frac{h_{1 i k}^2}{\mu_1 - \mu_k} & \geq & - \frac{3-3^{-9}}{2}
    \sqrt{\frac{n - 1}{n}} \sumik h_{1 i k}^2 \nonumber\\
    & \geq & - \frac{1-3^{-10}}{2} \sqrt{\frac{n - 1}{n}} \left( 3 \sumik h_{1
    i k}^2 + \sum_{i, j, k \geq 2} h_{i j k}^2 + h_{111}^2 \right) \\
    & = & - \frac{1-3^{-10}}{2} \sqrt{\frac{n - 1}{n}} | \nabla h |^2 . \nonumber
  \end{eqnarray}
  By Lemma \ref{idealp} and the definition of $\delta$, we have
  \begin{eqnarray}
    &  & \sqrt{\tfrac{n - 1}{n}} (n (H^2 + c) - \mathring{S}) + (n - 1) (\eta - 2)
    H \sqrt{\mathring{S}} \nonumber\\
    & \geq & \sqrt{\tfrac{n - 1}{n}} (n (H^2 + c) - \mathring{S}) - 2 (n - 1)
    H \sqrt{\mathring{S}} \nonumber\\
    & \geq & \sqrt{\tfrac{n - 1}{n}} (n (H^2 + c) - \mathring{\alpha} - \delta) - 2
    (n - 1) H \left( \sqrt{\mathring{\alpha}} + \frac{\delta}{2
    \sqrt{\mathring{\alpha}}} \right) \nonumber\\
    & = & - n H \sqrt{\mathring{\alpha}} - \sqrt{\tfrac{n - 1}{n}} \delta - (n -
    1) \frac{H \delta}{\sqrt{\mathring{\alpha}}}  \label{sqrtn-1/n}\\
    & \geq & - n H \sqrt{\mathring{\alpha}} - B_n H
    \sqrt{\mathring{\alpha}} - B_n(n - 1)H \sqrt{\mathring{\alpha}} \nonumber\\
    & \geq & - \frac{6}{5} n H \sqrt{\mathring{S}} . \nonumber
  \end{eqnarray}
  Combining {\eqref{Laph11=}}--{\eqref{sqrtn-1/n}}, we obtain
  \begin{equation}
    \Delta \lambda_1 \geq 3^{-10} \sqrt{\frac{n - 1}{n \mathring{S}}} |
    \nabla h |^2 + \sqrt{n (n - 1)} H \mathring{S} \phi - \frac{6}{5} n \eta H
    \mathring{S} . \label{1/33sqrt}
  \end{equation}
  It follows from Lemma \ref{phi<1/16sqrt} that $3 \eta + 2 \sigma <
  \frac{3}{4}$. Using Lemma \ref{mui} (ii), we get
  \begin{equation}
    \sqrt{n (n - 1)} \phi \geq \eta [3 n - (n + 1) (3 \eta + 2 \sigma)]
    \geq 2 n \eta . \label{1/2sqrt}
  \end{equation}
  By {\eqref{1/33sqrt}} and {\eqref{1/2sqrt}}, we obtain
  \begin{equation}
    \Delta \lambda_1 \geq 3^{-10} \sqrt{\frac{n - 1}{n \mathring{S}}} |
    \nabla h |^2 + \frac{2}{5} \sqrt{n (n - 1)} H \mathring{S} \phi \geq 0.
    \label{sqrtnn-1H}
  \end{equation}
  Since $S$ is constant, the principal curvatures and the sectional curvature
  of $M$ are bounded. Thus we can use the Omori-Yau maximum principle \cite{CY1975,Omori1967,Yau1975}: there
  exists a sequence of points $\{ x_k \} \subset M$, such that $\underset{k
  \rightarrow \infty}{\lim} \Delta \lambda_1 (x_k) \leq 0$. From
  {\eqref{sqrtnn-1H}}, we obtain
  \begin{equation}
    \lim_{k \rightarrow \infty} | \nabla h | (x_k) = 0 \qquad \operatorname{and}
    \qquad \lim_{k \rightarrow \infty} \phi (x_k) = 0. \label{limk->inf}
  \end{equation}
  From {\eqref{dh2+nHS}} and {\eqref{limk->inf}}, we have
  \[ | \nabla h | \equiv \phi \equiv 0 \]
  and
  \[ \mathring{S} - n (H^2 + c) + (n - 2) \sqrt{\frac{n}{n - 1}  \mathring{S}} H = 0. \]
  By Lemma \ref{idealp}, we get
  \[ \mathring{\alpha} - n (H^2 + c) + (n - 2) \sqrt{\frac{n}{n - 1}  \mathring{\alpha}}
     H = 0. \]
  Comparing the above two equations, we obtain $\mathring{S} = \mathring{\alpha}$.

  Since $\phi \equiv 0$, we get $\eta \equiv \sigma \equiv 0$. Therefore, $M$
  is an isoparametric hypersurface with two distinct principal curvatures, and
  the smallest principal curvature has multiplicity one.
\end{proof*}

\noindent Juan-Ru Gu\\
Center of Mathematical Sciences, \\Zhejiang University, \\Hangzhou 310027, China;\\
Department of Applied Mathematics, \\Zhejiang University of Technology, \\Hangzhou
310023, China\\
E-mail address: gujr@zju.edu.cn\\

\noindent Li Lei \\
Center of Mathematical Sciences, \\Zhejiang University, \\Hangzhou 310027, China\\
E-mail address: lei-li@zju.edu.cn\\

\noindent Hong-Wei Xu\\
Center of Mathematical Sciences, \\Zhejiang University, \\Hangzhou 310027, China\\
E-mail address: xuhw@zju.edu.cn\\

\medskip


\begin{thebibliography}{10}

  \bibitem{Chang1993b} S. P. Chang, A closed hypersurface of constant scalar
    and mean curvatures in $S^4$ is isoparametric, \emph{Comm. Anal.
    Geom.}, {\bf 1}(1993), 71-100.

  \bibitem{Chang1993} S. P. Chang, On minimal hypersurfaces with constant scalar
  curvatures in $S^4$, \emph{J. Differential Geom.}, {\bf 37}(1993),
  523-534.

  \bibitem{Chang1994} S. P. Chang, On closed hypersurfaces of constant scalar curvatures and mean curvatures
  in $S^{n+1}$, \emph{Pacific J. Math.}, {\bf 165}(1994), 67-76.

  \bibitem{Cheng1997} S. Y. Cheng, On the Chern conjecture for minimal hypersurface with constant scalar curvatures in the spheres, Tsing Hua Lectures on Geometry and Analysis, International Press,
  Cambridge, MA, 1997, 59-78.

  \bibitem{CY1975}  S.Y. Cheng, S.T. Yau, Differential equations on Riemannian manifolds and their geometric applications. \emph{Commun. Pure Appl. Math.}, {\bf 28}(1975), 333-354.


  \bibitem{CI1999} Q. M. Cheng and S. Ishikawa, A characterization of the Clifford torus,
  \emph{Proc. Amer. Math. Soc.}, {\bf 127}(1999), 819-828.

  \bibitem{CN1990} Q. M. Cheng and H. Nakagawa, Totally umbilic hypersurfaces, \emph{Hiroshima Math. J.},
  {\bf 20}(1990), 1-10.

  \bibitem{CW1994} Q. M. Cheng and Q. R. Wan, Complete hypersurfaces of $R^4$ with constant mean curvature, \emph{Monatsh. Math.}, {\bf 118}(1994), 171-204.

  \bibitem{Chern1968} S. S. Chern, Minimal submanifolds in a Riemannian manifold. University of Kansas, Department of Mathematics Technical Report 19,
  Univ. of Kansas, Lawrence, Kan., 1968.

  \bibitem{CdK1970} S. S. Chern, M. do Carmo and S. Kobayashi, Minimal submanifolds
  of a sphere with second fundamental form of constant length,
  \emph{Functional Analysis and Related Fields}, Springer-Verlag,
  Berlin, 1970, 59-75.

  \bibitem{dB1990} S. C. de Almeida and F. G. B. Brito, Closed 3-dimensional hypersurfaces with constant
  mean curvature and constant scalar curvature, \emph{Duke Math. J.}, {\bf 61}(1990), 195-206.


  \bibitem{DGW2017} Q. T. Deng, H. L. Gu and Q. Y. Wei, Closed Willmore minimal hypersurfaces with constant scalar curvature in  $\mathbb{S}^{5}(1)$ are isoparametric,  \emph{Adv. Math.}, {\bf 314}(2017), 278-305.

  \bibitem{DX2011} Qi Ding and Y. L. Xin, On Chern's problem for rigidity of minimal hypersurfaces in the spheres,
  \emph{Adv. Math.}, {\bf 227}(2011), 131-145.

  \bibitem{GT2012} J. Q. Ge and Z. Z. Tang, Chern conjecture and isoparametric hypersurfaces, Differential
  geometry, Adv. Lect. Math., 22, International Press, Somerville, MA, 2012, 49-60.

  \bibitem{GXXZ2016} J. R. Gu, H. W. Xu, Z. Y. Xu and E. T. Zhao, A survey on rigidity problems in geometry and topology of submanifolds,
  Proceedings of the 6th International Congress of Chinese Mathematicians, \emph{Adv. Lect. Math.}, {\bf
  37}, Higher Education Press \& International Press, Beijing-Boston,
  2016, 79-99.

  \bibitem{Lawson1969} B. Lawson, Local rigidity theorems for minimal hypersurfaces,
  \emph{Ann. of Math.}, {\bf 89}(1969), 187-197.

  \bibitem{LXX2017} L. Lei, H. W. Xu and Z. Y. Xu, On Chern's conjecture for minimal hypersurfaces in spheres, arXiv:1712.01175v1.

  \bibitem{LL1992} A. M. Li and J. M. Li, An intrinsic rigidity theorem for minimal submanifolds in a sphere,
  \emph{Arch. Math. (Basel)}, {\bf 58}(1992), 582-594.

  \bibitem{Munzner1980} H. F. M\"{u}nzner, Isoparametrische Hyperfl\"{a}chen in Sph\"{a}ren, I, II, Math. Ann., 251,
  256(1980, 1981), 57-71, 215-232.

  \bibitem{Muto1988} H. Muto, The first eigenvalue of the Laplacian of an isoparametric minimal hypersurface
  in a unit sphere, \emph{Math. Z.}, {\bf 197}(1988), 531-549.

\bibitem{Nunez} R. A. N$\acute{u}\tilde{n}$ez, On complete hypersurfaces with constant mean and scalar curvatures in Euclidean spaces, \emph{Proc. Amer. Math. Soc.}, {\bf 145}(2017),2677-2688
  \bibitem{Omori1967} H. Omori, Isometric immersions of Riemannian manifolds. \emph{J. Math. Soc. Japan}, {\bf 19}(1967), 205-214.

  \bibitem{PT1983a} C. K. Peng and C. L. Terng, Minimal hypersurfaces of sphere with constant scalar
  curvature, \emph{Ann. of Math. Stud.}, {\bf 103}, Princeton Univ.
  Press, Princeton, NJ, 1983, 177-198.

  \bibitem{PT1983b} C. K. Peng and C. L. Terng, The scalar curvature of minimal hypersurfaces in
  spheres, \emph{Math. Ann.}, {\bf 266}(1983), 105-113.

  \bibitem{SWY} M. Scherfner, S. Weiss and S. T. Yau, A review of the Chern conjecture for isoparametric hypersurfaces in spheres,
  Advances in Geometric Analysis, \emph{Adv. Lect. Math.}, {\bf 21}, International Press, Somerville, MA, 2012, 175-187.

  \bibitem{Simons1968} J. Simons, Minimal varieties in Riemannian manifolds, \emph{Ann. of Math.}, {\bf 88}(1968),
  62-105.

  \bibitem{SY2007} Y. J. Suh and H. Y. Yang, The scalar curvature of minimal
  hypersurfaces in a unit sphere, \emph{Commun. Contemp. Math.}, {\bf
  9}(2007), 183-200.

  \bibitem{TWY2018} Z. Z. Tang, D. Y. Wei and W. J. Yan, A sufficient condition for a hypersurface to be isoparametric, arXiv:1803.10006v1.

  \bibitem{WX2007} S. M. Wei and H. W. Xu, Scalar curvature of minimal hypersurfaces
  in a sphere, \emph{Math. Res. Lett.}, {\bf 14}(2007), 423-432.

  \bibitem{Xu1990} H. W. Xu, Pinching theorems, global pinching theorems and eigenvalues for Riemannian
    submanifolds, Ph.D. dissertation, Fudan University, 1990.

  \bibitem{Xu1993a} H. W. Xu, A rigidity theorem for submanifolds with parallel mean curvature in a sphere, \emph{Arch. Math. (Basel)}, {\bf 61}(1993),
  489-496.


  \bibitem{XT2011} H. W. Xu and L. Tian, A new pinching theorem for closed
  hypersurfaces with constant mean curvature in $S^{n+1}$, \emph{Asian
  J. Math.}, {\bf 15}(2011), 611-630.

  \bibitem{XX2013} H. W. Xu and Z. Y. Xu, The second pinching theorem for hypersurfaces with constant mean curvature in a
  sphere, \emph{Math. Ann.}, {\bf 356}(2013), 869-883.

  \bibitem{XX2014} H. W. Xu and Z. Y. Xu, A new characterization of the Clifford torus via scalar curvature pinching,
  \emph{J. Funct. Anal.}, {\bf 267}(2014), 3931-3962.

  \bibitem{XX2017} H. W. Xu and Z. Y. Xu, On Chern's conjecture for minimal hypersurfaces and rigidity of self-shrinkers, \emph{J. Funct. Anal.}, {\bf 273}(2017), 3406-3425.



  \bibitem{YC1990} H. C. Yang and Q. M. Cheng, A note on the pinching constant of minimal hypersurfaces with
  constant scalar curvature in the unit sphere, \emph{Kexue Tongbao},
  {\bf 35}(1990), 167-170; \emph{Chinese Sci. Bull.}, {\bf 36}(1991),
  1-6.

  \bibitem{YC1994} H. C. Yang and Q. M. Cheng, An estimate of the pinching constant of minimal hypersurfaces
  with constant scalar curvature in the unit sphere, \emph{Manuscripta
  Math.}, {\bf 84}(1994), 89-100.

  \bibitem{YC1998} H. C. Yang and Q. M. Cheng, Chern's conjecture on minimal hypersurfaces,
  \emph{Math. Z.}, {\bf 227}(1998), 377-390.

  \bibitem{Yau1975} S. T. Yau, Harmonic functions on complete Riemannian manifolds. \emph{Comm. Pure Appl. Math.}, {\bf 28}(1975), 201-228.

  \bibitem{Yau1982} S. T. Yau, Problem section, In: Seminar on Differential Geometry, \emph{Ann. of Math. Stud.}, {\bf 102},
  Princeton Univ. Press, Princeton, NJ, 1982, 669-706.

  \bibitem{Zhang2010} Q. Zhang, The pinching constant of minimal hypersurfaces in the unit
  spheres, \emph{Proc. Amer. Math. Soc.}, {\bf 138}(2010), 1833-1841.



  \end{thebibliography}
\end{document}